\documentclass[]{article}
\usepackage{amsmath,amssymb,amsthm}
\usepackage[latin1]{inputenc}
\usepackage{amsfonts}
\usepackage[usenames,dvipsnames]{pstricks}
\usepackage[all]{xy}
\usepackage{hyperref}
\usepackage{tikz}
\newtheorem{theorem}{Theorem}[section]

\newtheorem{proposition}[theorem]{Proposition}

\theoremstyle{definition}
\newtheorem{example}[theorem]{Example}
\newtheorem{definition}[theorem]{Definition}

\newtheorem*{remark}{Remark}

\begin{document}
\title{\bf Properties of  Bipolar Fuzzy Hypergraphs } \normalsize
\author{{\bf M. Akram$^{\bf a}$ ,  W.A. Dudek$^{\bf b}$ \, and \ S. Sarwar$^{\bf c}$} \\
{\small {\bf a.}  Punjab University College of Information
Technology, University of the Punjab,}\\
{\small  Old Campus, Lahore-54000, Pakistan.}\\
  {\small E-mail: makrammath@yahoo.com,
  ~~~~m.akram@pucit.edu.pk}\\
    {\small {\bf b.}  Institute of Mathematics and Computer Science,  Wroclaw University
of Technology,}\\
 {\small  Wyb. Wyspianskiego 27, 50-370,Wroclaw, Poland.}\\
  {\small E-mail: Wieslaw.Dudek@pwr.wroc.pl }\\
    {\small {\bf c.}  Punjab University College of Information
Technology, University of the Punjab,}\\
{\small  Old Campus, Lahore-54000, Pakistan.}\\
  {\small E-mail:  s.sarwar@pucit.edu.pk } }
 %\date{}
 \maketitle
 \hrule
\begin{abstract}
In this article,
 we apply the concept of bipolar fuzzy sets to hypergraphs and investigate some properties of bipolar fuzzy hypergraphs.  We introduce the notion of  $A-$ tempered bipolar fuzzy hypergraphs and present some of their properties. We also present application examples of  bipolar  fuzzy hypergraphs.
\end{abstract}
 {\bf Keywords}: Bipolar fuzzy hypergraph,  bipolar fuzzy partition, dual bipolar fuzzy hypergraph, $A-$ tempered bipolar fuzzy hypergraphs, clustering problem.\\
 {\bf  Mathematics Subject Classification 2010}: 05C99\\\\
\hrule
\section{Introduction}
  In 1994, Zhang \cite{ZW1}  initiated  the concept of  bipolar fuzzy sets   as a generalization of fuzzy sets \cite{LA}. Bipolar fuzzy sets are an extension of fuzzy sets whose membership degree range is  $[-1, 1]$. In a bipolar fuzzy set, the membership
degree $0$ of an element means that the element is irrelevant to
the corresponding property, the membership degree $(0,1]$ of an
element indicates that the element somewhat satisfies the
property, and the membership degree $[-1,0)$ of an element
indicates that the element somewhat satisfies the implicit
counter-property. Although bipolar fuzzy sets and intuitionistic
fuzzy sets look similar to each other, they are essentially
different  sets \cite{Lee2}. In many domains, it is important to be able to deal with
bipolar information. It is noted that positive information
represents what is granted to be possible, while negative
information represents what is considered to be impossible. This
domain has recently motivated new research in several directions.
In particular, fuzzy and possibilistic formalisms for bipolar
information have been  proposed \cite{DSH}, because when
we deal with spatial information in image processing or in spatial
reasoning applications, this bipolarity also occurs. For instance,
when we assess the position of an object in a space, we may have
positive information expressed as a set of possible places and
negative information expressed as a set of impossible places. \\
At present, graph theoretical concepts are highly utilized by computer science applications. Especially in research areas of computer science including data mining, image segmentation, clustering, image capturing and networking, for example a data structure can be designed in the form of tree which in turn utilized vertices and edges. Similarly,  modeling of network topologies can be done using graph concepts. In the same way the most important concept of graph coloring is utilized in resource allocation, scheduling. Also, paths, walks and circuits in graph theory are used in tremendous applications say traveling salesman problem, database design concepts, resource networking. This leads to the development of new algorithms and new theorems that can be used in tremendous applications.
Hypergraphs are the generalization of graphs (cf. \cite{CB}) in case of set of multiarity relations. It means the expansion of graph models for the modeling complex systems. In case of modeling systems with fuzzy binary and multiarity relations between objects, transition to fuzzy hypergraphs, which combine advantages both fuzzy and graph models, is more natural. It allows to realise formal optimization and logical procedures. However, using of the fuzzy graphs and hypergraphs as the models of various systems (social, economic systems, communication networks and others) leads to difficulties. The graph isomorphic transformations are reduced to redefinition of vertices and edges. This redefinition does not change properties the graph determined by an adjacent and an incidence of its vertices and edges. Fuzzy independent set, domination fuzzy set, fuzzy chromatic set are invariants concerning the isomorphism transformations of the fuzzy graphs and fuzzy hypergraph and allow make theirs structural analysis \cite{CB1}.
 Lee-kwang {\it et al.} \cite{Lee} generalized and redefined the concept of fuzzy hypergraphs whose basic idea was given by Kaufmann \cite{aa}. Further, the concept of fuzzy  hypergraphs was discussed in \cite{RH}.                                                                     Chen \cite{cc} introduced  the concept of interval-valued fuzzy hypergraphs. Parvathi {\it et al.}\cite{RM} introduced the concept of  intuitionistic fuzzy hypergraphs.  Samanta and  Pal [20] introduced the concept of  a bipolar fuzzy hypergraph  and studied some of its elementary properties.  In this article,  we first
           investigate some interesting  properties of bipolar fuzzy hypergraphs. We introduce the regularity of bipolar fuzzy hypergraphs.  We then introduce the notion of  $A-$ tempered bipolar fuzzy hypergraphs and present some of their properties. Finally, we  present an example of a bipolar  fuzzy partition on the digital image processing.  \\
           We  used standard  definitions and terminologies in this
paper. For  notations, terminologies and applications are not
mentioned  in the paper, the readers are referred to [1-9].

\section{Preliminaries}
A {\it hypergraph} is a pair $H^* = (V, E^*$), where $V$ is a finite set of nodes (vertices) and $E^*$ is a set of edges (or hyperedges) which are arbitrary nonempty subsets of $V$  such that $\bigcup_j E^*_j=V$.
A hypergraph is a generalization of an ordinary undirected graph, such that an edge need
not contain exactly two nodes, but can instead contain an arbitrary nonzero number of vertices.
An ordinary undirected graph (without self-loops) is, of course, a hypergraph where every edge
has exactly two nodes (vertices). A hypergraph is {\it simple} if there are no repeated edges and no edge properly contains another. Hypergraphs are often defined by an incidence matrix with columns indexed by
the edge set and rows indexed by the vertex set. The {\it rank} $r(H)$ of a hypergraph is defined as the
maximum number of nodes in one edge, $r(H)$ =$ \max_j( |Ej |)$, and the {\it anti-rank}
$s(H)$ is defined likewise, i.e., $s(H)$ =$ \min_j(|Ej|)$. We say that a hypergraph is {\it uniform} if $r(H)$ = $s(H)$. A uniform hypergraph of rank $k$ is called $k$-uniform hypergraph. Hence a simple
graph is a 2-uniform hypergraph, and thus all simple graphs are also hypergraphs. A hypergraph is {\it vertex}
(resp. {\it hyperedge}) {\it symmetric} if for any two vertices (resp. hyperedges) $v_i$ and $v_j$ (resp. $e_i$ and $e_j$), there is an automorphism of the hypergraph that maps $v_i$ to $v_j$ (resp. $e_i$ to $e_j$).
The {\it dual of a hypergraph} $H^* = (V,E^*)$ with vertex set $V = \{v_1, v_2, \ldots, v_n \}$ and hyperedge set $E^* =\{e^*_1, e^*_2, \ldots, e^*_m \}$ is a hypergraph $H^d = (V^d,
 E^*{^d})$ with vertex set $V^d = \{v^d_1, v^d_2, \ldots, v^d_m \}$ and hyperedge set
$E^d =\{(e^*_1)^d, (e^*_2)^d, \ldots, (e^*_n)^d \}$ such that $v^d_ j$ corresponds to $e^*_j$
with hyperedges $(e^*_i)^d = \{v^d_j ~|~v_i \in e^*_j$ and $e^*_j \in E^* \}$. In
other words, $H^d$ is obtained from $H^*$ by interchanging
of vertices and hyperedges in $H^*$. The incidence matrix
of $H^d$ is the transpose of the incidence matrix of $H^*$.Thus,  $(H^d)^d=H^*$.
\begin{definition}\cite{LA, LA1}
  A {\it fuzzy set} $\mu$ on a nonempty set $X$ is a map $\mu:X\to [0,1]$. In the clustering, the fuzzy set $\mu$, is called a {\it fuzzy class}. We define the
{\it support} of $\mu$ by supp ($\mu) = \{x \in X ~|~ \mu(x)\neq 0 \}$ and say $\mu$ is nontrivial if supp($\mu$)
is nonempty. The {\it height} of $\mu$ is $h(\mu)= \max \{\mu(x)~|~x \in X \}$. We say $\mu$ is
{\it normal} if $h(\mu)$ =1. A map $\nu: X\times X\to [0,1]$ is called  a {\it
fuzzy relation} on $X$ if $\nu(x,y)\leq \min(\mu(x),\mu(y))$ for
all $x,y\in X$.  A {\it fuzzy partition} of a set $X$ is a family of nontrivial fuzzy sets $\{\mu_1, \mu_2, \mu_3, \ldots, \mu_m \}$ such that
\begin{itemize}
  \item [\rm (1)]  $\bigcup_i {\rm supp}(\mu_i)=X$, ~~ $i=1, 2, \ldots, m$
    \item [\rm (2)] $\sum^m_{i=1} \mu_i(x)=1$  for all $x \in X$.
\end{itemize}
 We call a family $\{\mu_1, \mu_2, \mu_3, \ldots, \mu_m \}$ a {\it fuzzy covering} of $X$ if it verifies only the above conditions (1) and (2).
\end{definition}
\begin{definition}\cite{Lee2}
Let $V$ be a finite set and let $E$ be a finite family of
nontrivial fuzzy sets on $V$ such that  $V = \bigcup_{j}  {\rm supp}(\mu_j)$, where $\mu_j$ is membership function
defined on  $E_j \in E$. Then the pair
$H=(V, E)$ is a {\it fuzzy hypergraph} on $V$, $E$ is the family of fuzzy edges of $H$
and $V$ is the (crisp) vertex set of $H$.
\end{definition}
\begin{definition}\cite{Lee1, ZW1}~Let $X$ be a nonempty set. A
{\sl bipolar fuzzy set} $B$ in $X$ is an object having the form
$$B =\{(x,\,\mu^P (x),\,\mu^N (x))\,|\,x\in X\}$$ where $\mu^P :X\rightarrow [0,\,1]$ and $\mu^N :X\rightarrow
[-1,\,0]$ are  mappings.
\end{definition}
\par We use the positive membership degree $\mu^P (x)$ to denote the
satisfaction degree of an element $x$ to the property corresponding to a bipolar fuzzy set $B$, and the negative membership degree $\mu^N (x)$ to denote the satisfaction degree of an element $x$ to some explicit or implicit property corresponding to a bipolar fuzzy set $B$. If $\mu^P (x)\not=0$ and $\mu^N (x)=0$, it is the situation that $x$ is regarded as having only positive satisfaction for $B$. If $\mu^P (x)=0$ and $\mu^N (x)\not=0$, it is the situation that $x$ does not satisfy the property of $B$ but somewhat satisfies the counter property of $B$ . It is possible for an element $x$ to be such that $\mu^P(x)\not=0$ and $\mu^N(x)\not=0$ when the membership function of the property overlaps that of its counter property over some portion of $X$.\\
For the sake of simplicity, we shall use the symbol $B =(\mu^P ,\,\mu^N )$ for the bipolar fuzzy set $B =\{(x,\,\mu^P (x),\,\mu^N (x))\,|\,x\in X\}$.
%\begin{definition}\cite{Lee1}
%For every two bipolar  fuzzy sets $A=(\mu^P_A, \mu^N_A)$ and
%$B=(\mu^P_B,\mu^N_B)$ in $X$, we define
%\begin{itemize}
%\item $(A\bigcap B)(x)=(\min(\mu^P_A(x), \mu^P_B(x)), \max (\mu^N_A (x), \mu^N_B(x)))$,
%\item $(A\bigcup B)(x)=(\max(\mu^P_A(x), \mu^P_B(x)), \min (\mu^N_A (x), \mu^N_B(x)))$.
%\end{itemize}
%\end{definition}
\begin{definition}\cite{ZW1}
Let $X$ be a nonempty set. Then we call a  mapping $A=(\mu^P_A,
\mu^N_A): X \times X \to [0, 1] \times [-1, 0]$ a {\it bipolar
fuzzy relation} on $X$ such that $\mu^P_A(x, y) \in [0, 1]$ and $\mu^N_A(x, y) \in [-1, 0]$.
\end{definition}
%\begin{definition}\cite{ZW1}
%Let  $A=(\mu^P_A, \mu^N_A)$ and $B=(\mu^P_B, \mu^N_B)$ be bipolar
%fuzzy sets on a set $X$. If $A=(\mu^P_A, \mu^N_A)$ is a
%bipolar  fuzzy relation on a set $X$, then $A=(\mu^P_A,
%\mu^N_A)$ is called a {\it bipolar  fuzzy relation} on
%$B=(\mu^P_B, \mu^N_B)$ if $ \mu^P_A(x, y) \leq \min(\mu^P_B(x), \mu^P_B(y))$
%and $\mu^N_A(x, y)\geq \max( \mu^N_B(x), \mu^N_B(y))$ for all $x$, $y$
%$\in X$. A bipolar  fuzzy relation  $A$ on $X$ is called
%{\it symmetric} if $\mu^P_A(x, y)= \mu^P_A(y, x)$ and $\mu^N_A(x, y)=
%\mu^N_A(y, x)$ for all $x$, $y$ $\in X$.
%\end{definition}
\begin{definition}{\rm \cite{Lee1}}
The {\it support} of  a bipolar fuzzy set $A=(\mu^P_A, \mu^N_A)$, denoted by supp($A$), is defined by
\[{\rm supp}(A)=supp^P(A) \cup supp^N(A),~ supp^P(A)= \{ x\,|\,\mu^P_A(x) > 0 \},~~ supp^N(A)= \{ x\,|\,\mu^N_A(x) < 0 \}.\]
We call supp$^P(A)$  as positive support and supp$^N(A)$  as negative support.
\end{definition}
\begin{definition}{\rm \cite{Lee1}}
Let $A=(\mu^P_A,\mu^N_A)$ be a bipolar fuzzy set on $X$ and let $\alpha \in [0, 1]$.   $\alpha$-cut $A_\alpha$ of $A$   can be  defined  as
 \[A_{\alpha}= A^P_{\alpha} \cup A^N_{\alpha},~ A^P_{\alpha}= \{x ~|~ \mu^P_\alpha(x)\ge \alpha \}, ~A^P_{\alpha}= \{x ~|~ \mu^N_\alpha(x)\le -\alpha \}.    \]
We call $A^P_{\alpha} $ as positive  $\alpha$-cut and $A^N_{\alpha} $ as negative  $\alpha$-cut.
\end{definition}
\begin{definition}
The {\it height} of a bipolar fuzzy set $A=(\mu^P_A, \mu^N_A)$ is  defined as $h(A)=  \max \{\mu^P_A(x) | x \in X\}.$ The {\it depth} of a bipolar fuzzy set $A=(\mu^P_A, \mu^N_A)$ is  defined as $d(A)=  \min \{\mu^N_A(x) | x \in X\}.$ We shall say that bipolar fuzzy set $A$ is
{\it normal}, if there is at least one $x \in X$ such that $\mu^P_A(x)$ =1 or $\mu^N_A(x)$ =$-1$.
\end{definition}

\section{ Bipolar fuzzy hypergraphs}
\begin{definition} {\rm \cite{SM}}
Let $V$ be a finite set and let $E=\{E_1, E_2, \ldots, E_m\}$  be a finite family of
nontrivial bipolar  fuzzy subsets of $V$ such that
\[
 V = \bigcup_{j}  {\rm supp}(\mu^P_j, \mu^N_j),~~ j=1, 2 \ldots, m,
\]
where  $\mu^P_j$ and $\mu^N_j$ are positive and negative  membership functions defined on $E_j \in E.$

Then the pair
$H=(V, E)$ is a {\it bipolar  fuzzy hypergraph} on $V$, $E$ is the family of bipolar  fuzzy edges of $H$
and $V$ is the (crisp) vertex set of $H.$ The order of $H$ (number of vertices) is
denoted by $|V|$ and the number of edges is denoted by $|E|$.
\end{definition}
Let $A=(\mu^P_A, \mu^N_A)$ be a bipolar fuzzy subset of $V$ and let $E$ be a collection of bipolar fuzzy
subsets of $V$ such that for each $B=(\mu^P_B, \mu^N_B) \in E$ and $x \in V$, $\mu^P_B(x)\leq \mu^P_A(x)$, $\mu^N_B(x)\geq \mu^N_A(x)$. Then the
pair $(A, B)$ is a bipolar fuzzy hypergraph on the bipolar fuzzy set $A$.
The bipolar fuzzy hypergraph $(A, B)$ is also a bipolar fuzzy hypergraph on $V$ = supp($A$),
the bipolar fuzzy set $A$ defines a condition for positive membership and negative membership in the edge set $E$. This condition can be stated separately, so without loss of generality
we restrict attention to bipolar fuzzy hypergraphs on crisp vertex sets.

\begin{example}
Consider a bipolar fuzzy hypergraph $H=(V, E)$ such that $V=\{a, b, c, d\}$ and $E=\{E_1, E_2, E_3\}$, where
\[
E_1=\{ \frac{a}{(0.2, -0.3)}, \frac{b}{(0.4, -0.5)}\},~E_2=\{  \frac{b}{(0.4, -0.5)}, \frac{c}{(0.5, -0.2)}\},~
E_3=\{ \frac{a}{(0.2, -0.3)}, \frac{d}{(0.2, -0.4)}\}.
\]
\begin{figure}[!h]
\centering

\scalebox{1} {
\begin{pspicture}(0,-2.4589062)(7.1028123,2.4189062)
\psellipse[linewidth=0.04,dimen=outer](5.2609377,0.14890625)(0.64,2.25)
\psellipse[linewidth=0.04,dimen=outer](1.6909375,0.14890625)(0.61,2.27)
\psellipse[linewidth=0.04,dimen=outer](3.4909375,1.54310626)(2.53,0.49)
%\psellipse[linewidth=0.04,dimen=outer](3.5309374,-1.5410937)(2.43,0.42)
\psdots[dotsize=0.12](1.5609375,1.54310626)
\psdots[dotsize=0.12](5.3609376,1.54310626)
\psdots[dotsize=0.12](5.3009377,-1.5610938)
\psdots[dotsize=0.12](1.5409375,-1.6010938)
\usefont{T1}{ptm}{m}{n}
%\rput(3.7723436,-2.2310936){$E_2$}
%\usefont{T1}{ptm}{m}{n}
\rput(6.592344,0.00890625){$E_2$}
\usefont{T1}{ptm}{m}{n}
\rput(0.43234375,0.10890625){$E_3$}
\usefont{T1}{ptm}{m}{n}
\rput(0.3423437,1.8089062){$a(0.2, -0.3)$}
\usefont{T1}{ptm}{m}{n}
\rput(0.12234375,-1.6310937){$d(0.2, -0.4)$}
\usefont{T1}{ptm}{m}{n}
\rput(3.6523438,2.2289062){$E_1$}
\usefont{T1}{ptm}{m}{n}
\rput(6.502344,2.1889062){$b(0.4, -0.5)$}
\usefont{T1}{ptm}{m}{n}
\rput(6.502344,-1.8710938){$c(0.5, -0.2)$}
\end{pspicture}
}
\caption{Bipolar fuzzy hypergraph}
\end{figure}
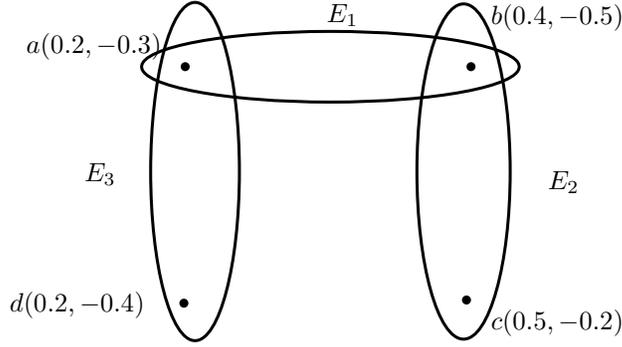

\begin{table}[!h]
\caption{The corresponding incidence matrix is given below: }
\centering
\medskip
\begin{tabular}{l|ccc}

$M_H$& $E_1$&$E_2$ &$E_3$\\
 \hline
a &$(0.2, -0.3)$ & (0, 0)&$(0.2, -0.3)$ \\
b& $(0.4, -0.5)$ & $(0.4, -0.5)$ & (0, 0)\\
c& (0, 0)& $(0.5, -0.2)$ &(0, 0)\\
d&(0, 0)&(0, 0)& $(0.2, -0.4)$\\
\end{tabular}
\end{table}

\end{example}

\begin{definition}

A bipolar fuzzy set $A=(\mu^P_A, \mu^N_A):X \to [0, 1] \times [-1, 0]$ is an {\it elementary  bipolar fuzzy set} if $A$ is single valued on  supp($A$).  An elementary  bipolar fuzzy hypergraph $H=(V, E)$ is a bipolar fuzzy hypergraph whose edges are  elementary.
\end{definition}
%\begin{definition}\cite{MA22}
%By a {\it bipolar fuzzy graph}  $G=(A,B)$ of a graph $G^*=(V,E)$ we
%mean a pair $G=(A,B)$, where $A=(\mu^P_A, \mu^N_A)$ is a
% bipolar fuzzy set on $V$ and $B=(\mu^P_B, \mu^N_B)$ is a
% bipolar fuzzy  relation on $E$
% such that \[ \mu^P_B(x y)\leq \min(\mu^P_B(x), \mu^P_A(y))~~~{\rm and}~~~\mu^N_{B}(x y)\geq \max(\mu^N_A(x), \mu^N_A(y))\]  for all $ xy    \in E$.  We note  $B$ is  symmetric in undirected graph,  but $B$ need not to be symmetric in digraph.
%\end{definition}
We explore the sense in which a bipolar fuzzy graph is a bipolar fuzzy hypergraph.
\begin{proposition}
Bipolar fuzzy graphs  are special cases of
the bipolar fuzzy hypergraphs.
\end{proposition}

A bipolar  fuzzy multigraph is a multivalued symmetric mapping $D=(\mu^P_D, \mu^N_D): V \times V \to [0, 1] \times [-1, 0]$.
A bipolar  fuzzy multigraph can be considered to be the ``disjoint union" or
``disjoint sum" of a collection of simple bipolar  fuzzy graphs, as is done with
crisp multigraphs. The same holds for multidigraphs. Therefore, these
structures can be considered as ``disjoint unions" or ``disjoint sums" of
bipolar  fuzzy hypergraphs.
\begin{definition}\label{D3.6}
A bipolar fuzzy hypergraph $H=(V, E)$ is {\it simple} if $A=(\mu^P_A, \mu^N_A)$, $B=(\mu^P_B, \mu^N_B) \in E$ and $\mu^P_A\leq  \mu^P_B $, $\mu^N_A\geq  \mu^N_B $ imply  that $\mu^P_A=\mu^P_B$, $\mu^N_A=\mu^N_B$. In particular, a (crisp) hypergraph $H^* =(V, E^*)$ is simple
if $X$, $Y\in E^*$ and $X \subseteq Y$ imply that $X = Y$.
A bipolar fuzzy hypergraph $H = (V, E)$ is {\it support simple} if $A=(\mu^P_A, \mu^N_A)$, $B=(\mu^P_B, \mu^N_B)$ $ \in E$,  supp($A$) = supp($B$), and $\mu^P_A\leq  \mu^P_B $, $\mu^N_A\geq  \mu^N_B $ imply  that $\mu^P_A=\mu^P_B$, $\mu^N_A=\mu^N_B$.
A bipolar fuzzy hypergraph $H = (V, E)$ is {\it strongly support simple} if $A=(\mu^P_A, \mu^N_A)$, $B=(\mu^P_B, \mu^N_B) \in E$ and  supp($A$) = supp($B$) imply that $A=B$.
\end{definition}
\begin{remark}
 The definition \ref{D3.6}  reduces to familiar definitions in
the special case where $H$ is a crisp hypergraph. The bipolar fuzzy definition of
simple is identical to the crisp definition of simple. A crisp hypergraph
is support simple and strongly support simple if and only if it has no
multiple edges. For bipolar fuzzy hypergraphs all three concepts imply no multiple edges.
Simple bipolar fuzzy hypergraphs are support simple and strongly support simple
bipolar fuzzy hypergraphs are support simple.
Simple and strongly support simple are independent concepts.
\end{remark}
\begin{definition}
 Let $H=(V, E)$ be a bipolar fuzzy hypergraph. Suppose that $\alpha \in [0, 1]$, $\beta \in $ [-1, 0]. Let
 \begin{itemize}
   \item $E_{(\alpha, \beta)}= \{A_{(\alpha, \beta)}|\,{\rm where ~A~ is~ positive~ and ~negative~ membership~ function~ defined~ on~ E_j\in E}\}, \ \ \ \ \ A_{(\alpha, \beta)}=\{x\,|\,\mu^P_A(x)\geq \alpha ~{\rm or}~ \mu^N_A(x)\leq  \beta\},$ \ and
   \item $V_{(\alpha, \beta)}=\bigcup_{A \in E}A_{(\alpha, \beta)}.$ \end{itemize}

 If $E_{(\alpha, \beta)} \neq \emptyset$, then the crisp hypergraph $H_{(\alpha, \beta)}=(V_{(\alpha, \beta)}, E_{(\alpha, \beta)})$ is the $(\alpha,\beta)-$ level hypergraph of $H$.
\end{definition}

Clearly, it is possible that $A_{(\alpha,\beta)}=B_{(\alpha,\beta)}$ for $A \neq B,$ by using distinct markers to
identity the various members of $E$ a distinction between  $A_{(\alpha,\beta)}$ and $B_{(\alpha,\beta)}$ to represent multiple edges in  $H_{(\alpha, \beta)}$. However, we do not take this approach unless otherwise stated, we will always regard  $H_{(\alpha, \beta)}$  as having no repeated edges.

The families of crisp sets (hypergraphs) produced by the $(\alpha, \beta)$-cuts of a bipolar fuzzy hypergraph share  an important relationship with each other, as expressed below:\\
suppose $\mathbb{X}$ and $\mathbb{Y}$ are two families of sets  such that for each set $X$ belonging to $\mathbb{X}$ there is at least one set $Y$ belonging to $\mathbb{Y}$ which contains $X$. In this case we say that $\mathbb{Y}$  {\it absorbs} $\mathbb{X}$ and symbolically write $\mathbb{X}\sqsubseteq \mathbb{Y}$ to express this relationship between $\mathbb{X}$ and $\mathbb{Y}$. Since it is possible for $\mathbb{X}\sqsubseteq \mathbb{Y}$ while $\mathbb{X}\cap \mathbb{Y} =\emptyset,$ we have that $\mathbb{X}\subseteq \mathbb{Y}\Rightarrow$ $\mathbb{X}\sqsubseteq \mathbb{Y}$,  whereas the converse is generally false. If $\mathbb{X}\sqsubseteq \mathbb{Y}$
and $\mathbb{X}\neq \mathbb{Y}$, then we write $\mathbb{X}\sqsubset \mathbb{Y}$.

\begin{definition}
Let $H=(V, E)$ be a bipolar fuzzy hypergraph.  Let $H_{(s, t)}$ be the $(s, t)-$ level hypergraph of $H$.  The sequence of real numbers
\[
\{(s_1, r_1), (s_2, r_2), \ldots, (s_n, r_n)\}, \ \ \  0 < s_1 < s_2 < \ldots < s_n ~{\rm and}~ 0 > r_1 > r_2 >\ldots
 > r_n, ~{\rm where}~  (s_n, r_n)=h(H),
\]
which satisfies the properties:
\begin{itemize}
  \item if \ $(s_{i-1}, r_{i-1})< (u, v) \leq (s_i, r_i)$, then $E_{(u, v)}=E_{(s_i, r_i)}$, and
  \item $E_{(s_i, r_i)}\sqsubset E_{(s_{i+1}, r_{i+1})}$,
\end{itemize}
is called the {\it fundamental sequence}
of $H,$ and is denoted by $F(H)$ and the set of $(s_i, r_i)$-level hypergraphs $\{H_{(s_1, r_1)}, H_{(s_2, r_2)}, \ldots,  H_{(s_n, r_n)}\}$ is called the {\it set of core hypergraphs} of $H$ or, simply, the {\it core set} of $H$, and is denoted by $C(H).$
\end{definition}

\begin{definition}
Suppose $H=(V, E)$ is a bipolar fuzzy hypergraph with
\[
F(H)=\{(s_1, r_1), (s_2, r_2), \ldots, (s_n, r_n)\},
\]
and $s_{n+1}=0,$ $r_{n+1}=0,$  then $H$ is called {\it sectionally elementary} if for each
edge $A=(\mu^P_A, \mu^N_A)\in E$, each $i=\{1, 2, \ldots, n\}$, and $(s_i, r_i) \in F(H)$, $A_{(s, t)}=A_{(s_i, r_i)}$ for all $(s, t) \in ((s_{i-1}, r_{i-1}), (s_{i}, r_{i}) ].$
\end{definition}
Clearly $H$ is sectionally elementary if and only if $A(x)=(\mu^P_A(x), \mu^N_A(x)) \in F(H)$ for
each $ A \in E$ and each $x \in X$.

\begin{definition}
A sequence of crisp hypergraphs $H_i=(V_i, E^*_i)$, $1\leq i \leq n$, is said
to be {\it ordered} if $ H_1 \subset H_2\subset \ldots \subset H_n$. The
sequence $\{H_i\,|\, 1\leq i \leq n \}$ is {\it simply ordered} if it is ordered and if whenever $E^* \in E^*_{i+1}-E^*_i$, then $E^* \nsubseteq V_i$.
\end{definition}
\begin{definition}
A bipolar  fuzzy hypergraph $H$ is {\it ordered} if the $H$ induced
fundamental sequence of hypergraphs is ordered. The bipolar  fuzzy hypergraph $H$ is
{\it simply ordered} if the $H$ induced fundamental sequence of hypergraphs is
simply ordered.
\end{definition}

\begin{example}
Consider the bipolar fuzzy hypergraph $H=(V, E)$, where $V=\{a, b, c, d\}$ and $E=\{E_1, E_2, E_3, E_4, E_5\}$ which is represented by the following incidence matrix:
\begin{table}[!h]
\caption{Incidence matrix of $H$}
\centering
\medskip
\begin{tabular}{l|ccccc}
$H$& $E_1$&$E_2$ & $E_3$ & $E_4$& $E_5$  \\
 \hline
$a$ &$(0.7, -0.2)$ & $(0.9, -0.2)$ & (0, 0) & (0, 0) & $(0.4, -0.3)$\\
$b$& $(0.7, -0.2)$& $(0.9, -0.2)$ & $(0.9, -0.2)$ & $(0.7, -0.2)$ & (0, 0)\\
$c$& (0, 0)&(0, 0) & $(0.9, -0.2)$ & $(0.7, -0.2)$ & $(0.4, -0.3)$\\
$d$& (0, 0)& $(0.4, -0.3)$ & (0, 0) & $(0.4, -0.3)$ & $(0.4, -0.3)$\\
\end{tabular}
\end{table}
\newline
Clearly, $h(H)=(0.9, -0.1).$\\
Now
\[ E_{(0.9, -0.1)}=\{\{a, b\}, \{ b, c\}   \}  \]
\[ E_{(0.7, -0.2)}=\{ \{a, b\}, \{ b, c\}\}\]
\[
E_{(0.4, -0.3)}=\{\{a, b\}, \{a, b, d\}, \{ b, c\}, \{ b, c, d\}, \{a, c, d\}   \}.
\]
Thus for $ 0.4 < s  \leq 0.9$ and $-0.1 > t \geq -0.3$,
 $E_{(s, t)}=\{ \{a, b\}, \{ b, c\}\}$, and for  $0 < s \leq 0.4$ and  $ -1 < t \geq -0.3$,

\[
E_{(s, t)}=\{\{a, b\}, \{a, b, d\}, \{ b, c\}, \{ b, c, d\}, \{a, c, d\}   \}.
\]
We note that
$ E_{(0.9, -0.1)} \subseteq E_{(0.4, -0.3)}$.
The fundamental sequence is $F(H)$=$\{(s_1, r_1)=(0.9, -0.1), \ (s_2, r_2)=(0.4, -0.3)\}$ and the set of core hypergraph is $C(H)=\{H_1=(V_1, E_1)=H_{(0.9, -0.1)}, \ H_2=(V_2, E_2)=H_{(0.4, -0.3)} \}$, where
\[
V_1=\{a, b, c\}, ~E_1=\{\{a, b\}, \{b, c\}\}
\]
\[
V_2=\{a, b, c, d\}, \ E_2=\{\{a, b\}, \{a, b, d\}, \{ b, c\}, \{b, c, d\}, \{ a, c, d\}\}.
 \]
$H$ is support simple, but not simple. $H$ is not sectionally elementary since $E_{1(s, t)}
  \neq E_{1(0.9, -0.1)}$ for $s=0.7$, $t=-0.2$. Clearly,  bipolar fuzzy hypergraph $H$ is simply ordered.
\end{example}

\begin{proposition}
Let $H=(V, E)$ be an elementary  bipolar fuzzy hypergraph. Then $H$ is
support simple if and only if $H$ is strongly support simple.
\end{proposition}
\begin{proof} Suppose that $H$ is elementary, support simple and that
supp($A$) = supp($B$). We assume without loss of generality that $h(A) \leq h(B)$. Since $H$ is  elementary,
it follows that $\mu^P_A \leq \mu^P_B$, $\mu^N_A \geq \mu^N_B$ and since $H$ is support simple that $\mu^P_A=\mu^P_B$, $\mu^N_A=\mu^N_B$. Therefore
$H$ is strongly support simple. The proof of converse part is obvious.
\end{proof}
The complexity of a bipolar fuzzy hypergraph depends in part on how many
edges it has. The natural question arises: is there an upper bound on the
number of edges of a bipolar fuzzy hypergraph of order $n$?
\begin{proposition}\label{P3.14}
Let $H=(V, E)$ be a simple bipolar fuzzy hypergraph of order $n$.
Then there is no upper bound on $|E|$.
\end{proposition}
\begin{proof}
Let $V = \{x,y\}$, and define $E_N$= $\{A_i= (\mu^P_{A_i}, \mu^N_{A_i})~|~i = 1,2,\ldots, N \}$,
where
\[
\mu^P_{A_i}(x)=\frac{1}{i+1}, ~~\mu^N_{A_i}(x)=-1 +\frac{1}{i+1},
\]
\[
\mu^P_{A_i}(y)=\frac{1}{i+1}, ~~\mu^N_{A_i}(y)=-\frac{i}{i+1}.
\]
Then $H_N = (V, E_N)$ is a simple bipolar fuzzy hypergraph with $N$ edges. This ends the proof.
\end{proof}
\begin{proposition}
Let $H=(V, E)$ be a support simple bipolar fuzzy hypergraph of order $n$.
Then there is no upper bound on $|E|$.
\end{proposition}
\begin{proof}
The class of support simple bipolar fuzzy hypergraphs contains the class of
simple bipolar fuzzy hypergraphs, thus the result follows from Proposition \ref{P3.14}.
\end{proof}

\begin{proposition}
Let $H=(V, E)$ be an elementary simple bipolar fuzzy hypergraph of order $n$.
Then there is no upper bound on $|E|\leq 2^n-1$ if and only if $\{supp(A)\,|\, A \in E\}=P(V)-\emptyset$.
\end{proposition}
\begin{proof}
Since $H$ is elementary and simple, each nontrivial $ W\subseteq  V$ can be the
support of at most one $A=(\mu^P_A, \mu^N_A) \in E$. Therefore, $|E|\leq 2^n-1$. To show there exists
an elementary, simple $H$ with $|E|= 2^n-1$, let $E=\{A=(\mu^P_A, \mu^N_A)~|~ W\subseteq V\}$ be the set of
functions defined by
\[
\mu^P_A(x)= \frac{1}{|W|}, ~{\rm if} ~x \in W, \ \ ~~~\mu^P_A(x)=0, \ {\rm if}~ x \notin W,
\]
\[
\mu^N_A(x)=-1 + \frac{1}{|W|}, ~{\rm if} ~x \in W, \ \  ~~~\mu^N_A(x)=-1, {\rm if}~ x \notin W.
\]

Then each one element has height $(1, -1)$, each two elements  has height $(0.5, -0.5)$ and so on. Hence $H$ is an elementary and simple, and $|E|= 2^n-1$.
\end{proof}

We state the following proposition without proof.
\begin{proposition}$\rule{20mm}{0mm}$\\
$(a)$ \ If $H=(V, E)$ is an elementary bipolar fuzzy hypergraph, then
$H$ is ordered. \\
$(b)$ \ If $H$ is an ordered bipolar fuzzy hypergraph with simple
support hypergraph, then $H$ is elementary.
\end{proposition}

\begin{definition}
The {\it dual of a bipolar fuzzy hypergraph} $H=(V, E)$ is a bipolar fuzzy
hypergraph $H^D = (E^D, V^D)$ whose vertex set is the edge set of $H$ and with edges
$V^D:E^D \to [0, 1] \times [-1, 0]$ by $V^D(A^D) = (\mu^D_A(x), \nu^D_A(x))$. $H^D$ is a bipolar fuzzy hypergraph whose incidence matrix is the transpose of the incidence matrix of $H$, thus $H^{DD} =H$.
\end{definition}
\begin{example}
Consider a bipolar fuzzy hypergraph $H = (V,E)$ such that $V = \{ x_1, x_2, x_3, x_4\}$,
$E = \{E_1, E_2, E_3, E_4 \}$, where $E_1 = \{\frac{x_1}{(0.5, -0.3)},  \frac{x_2}{(0.4, -0.2)}\}$, $E_2 =\{ \frac{x_2}{(0.4, -0.2)},  \frac{x_3}{(0.3, -0.6)}\}$, $E_3 = \{\frac{x_3}{(0.3, -0.6)}, \frac{x_4}{(0.5, -0.1)}\}$, $E_4 = \{\frac{x_4}{(0.5, -0.1)} , \frac{x_1}{(0.5, -0.3)}\}.$
\begin{figure}[!h]
\centering

\scalebox{1} {
\begin{pspicture}(0,-2.4589062)(7.1028123,2.4189062)
\psellipse[linewidth=0.04,dimen=outer](5.2609377,0.14890625)(0.64,2.25)
\psellipse[linewidth=0.04,dimen=outer](1.6909375,0.14890625)(0.61,2.27)
\psellipse[linewidth=0.04,dimen=outer](3.4909375,1.54310626)(2.53,0.49)
\psellipse[linewidth=0.04,dimen=outer](3.5309374,-1.5410937)(2.43,0.42)
\psdots[dotsize=0.12](1.5609375,1.54310626)
\psdots[dotsize=0.12](5.3609376,1.54310626)
\psdots[dotsize=0.12](5.3009377,-1.5610938)
\psdots[dotsize=0.12](1.5409375,-1.6010938)
\usefont{T1}{ptm}{m}{n}
\rput(3.7723436,-2.2310936){$E_3$}
\usefont{T1}{ptm}{m}{n}
\rput(6.592344,0.00890625){$E_2$}
\usefont{T1}{ptm}{m}{n}
\rput(0.43234375,0.10890625){$E_4$}
\usefont{T1}{ptm}{m}{n}
\rput(0.3423437,1.8089062){$x_1(0.5, -0.3)$}
\usefont{T1}{ptm}{m}{n}
\rput(0.12234375,-1.6310937){$x_4(0.5, -0.1)$}
\usefont{T1}{ptm}{m}{n}
\rput(3.6523438,2.2289062){$E_1$}
\usefont{T1}{ptm}{m}{n}
\rput(6.502344,2.1889062){$x_2(0.4, -0.2)$}
\usefont{T1}{ptm}{m}{n}
\rput(6.502344,-1.8710938){$x_3(0.3, -0.6)$}
\end{pspicture}
}
\caption{Bipolar fuzzy hypergraph}
\end{figure}
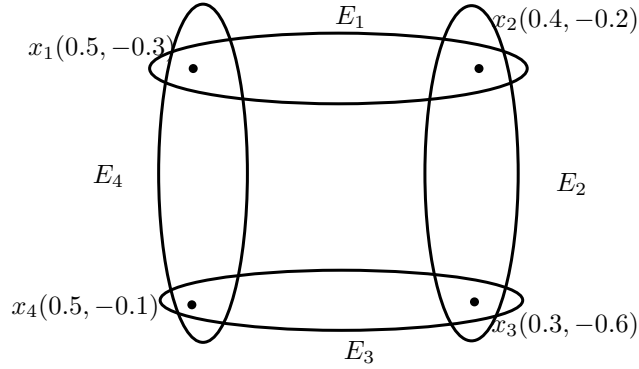

\begin{table}[!h]
\caption{The corresponding incidence matrix of $H$ is given below: }
\centering
\medskip

\begin{tabular}{l|cccc}
$M_H$& $E_1$&$E_2$ &$E_3$ & $E_4$\\
 \hline
$x_1$ &$(0.5, -0.3)$ & (0, 0)& (0, 0)& $(0.5, -0.3)$ \\
$x_2$& $(0.4, -0.2)$& $(0.4, -0.2)$& (0, 0)& (0, 0)\\
$x_3$& (0, 0)& $(0.3, -0.6)$& $(0.3, -0.6)$& (0, 0)\\
$x_4$&(0, 0)&(0, 0)& $(0.5, -0.1)$ & $(0.5, -0.1)$\\
\end{tabular}
\end{table}
\newpage
Consider the dual bipolar fuzzy hypergraph $H^D = (E^D, V^D)$ of $H$ such that $E^D = \{ e_1, e_2, e_3, e_4\}$,
$V^D = \{A, B, C, D \}$ where \[A = \{\frac{e_1}{(0.5, -0.3)},~  \frac{e_4}{(0.5, -0.3)}\},~ B =\{ \frac{e_1}{(0.4, -0.2)},  \frac{e_2}{(0.4, -0.2)}\},\] \[C = \{\frac{e_2}{(0.3, -0.6)},~ \frac{e_3}{(0.3, -0.6)}\},~ D = \{\frac{e_3}{(0.5, -0.1)}, ~ \frac{e_4}{(0.5, -0.1)}\}.\]

\begin{figure}[!h]
\centering

\scalebox{1} {
\begin{pspicture}(0,-2.4589062)(7.1028123,2.4189062)
\psellipse[linewidth=0.04,dimen=outer](5.2609377,0.14890625)(0.64,2.25)
\psellipse[linewidth=0.04,dimen=outer](1.6909375,0.14890625)(0.61,2.27)
\psellipse[linewidth=0.04,dimen=outer](3.4909375,1.54310626)(2.53,0.49)
\psellipse[linewidth=0.04,dimen=outer](3.5309374,-1.5410937)(2.43,0.42)
\psdots[dotsize=0.12](1.5609375,1.54310626)
\psdots[dotsize=0.12](5.3609376,1.54310626)
\psdots[dotsize=0.12](5.3009377,-1.5610938)
\psdots[dotsize=0.12](1.5409375,-1.6010938)
\usefont{T1}{ptm}{m}{n}
\rput(3.7723436,-2.2310936){$D$}
\usefont{T1}{ptm}{m}{n}
\rput(6.592344,0.00890625){$C$}
\usefont{T1}{ptm}{m}{n}
\rput(0.43234375,0.10890625){$A$}
\usefont{T1}{ptm}{m}{n}
\rput(0.3423437,1.8089062){$e_1$}
\usefont{T1}{ptm}{m}{n}
\rput(0.12234375,-1.6310937){$e_4$}
\usefont{T1}{ptm}{m}{n}
\rput(3.6523438,2.2289062){$B$}
\usefont{T1}{ptm}{m}{n}
\rput(6.502344,2.1889062){$e_2$}
\usefont{T1}{ptm}{m}{n}
\rput(6.502344,-1.8710938){$e_3$}
\end{pspicture}
}
\caption{Dual bipolar fuzzy hypergraph}
\end{figure}
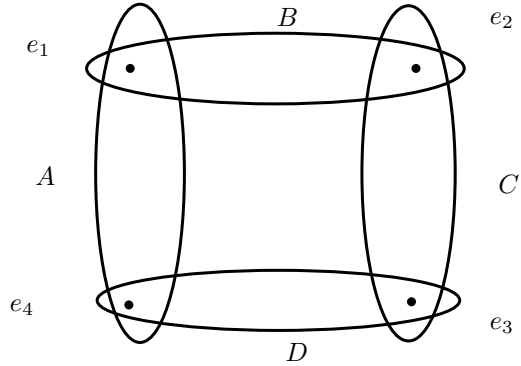
\begin{table}[!h]
\caption{The corresponding incidence matrix of $H^D$ is given below: }
\centering
\medskip
\begin{tabular}{l|cccc}
$M_{H^D}$& $A$&$B$ &$C$ & $D$\\
 \hline
$e_1$ &$(0.5, -0.3)$ & $(0.4, -0.2)$& (0, 0)&(0, 0) \\
$e_2$&(0, 0)& $(0.4, -0.2)$& $(0.3, -0.6)$& (0, 0)\\
$e_3$& (0, 0)&(0, 0)& $(0.3, -0.6)$& $(0.5, -0.1)$\\
$e_4$& $(0.5, -0.3)$&(0, 0)&(0, 0) & $(0.5, -0.1)$\\
\end{tabular}
\end{table}
\end{example}
We see that some edges contain only vertices having high positive membership
degree and high negative membership degree. We define  here the concept of strength of an edge.
\begin{definition}\label{D2}
The {\it strength} $\eta$ of an edge $E$ is the maximum positive membership $\mu^P(x)$ of
vertices  and maximum negative membership $\mu^N(x)$ of vertices in the edge $E$. That is,
$\eta(E_j)$ = $\{\max ( \mu^P_j(x)~|~ \mu^P_j(x)>0 ) ~{\rm ,}~ \max (\mu^N_j(x) ~|~ \mu^N_j(x) <0) \}$.
\end{definition}
Its interpretation is that the edge $E_j$ groups elements having
participation degree at least $\eta(E_j)$ in the hypergraph.

\begin{example}
Consider a bipolar fuzzy hypergraph $H = (V,E)$ such that $V = \{ a, b, c, d\}$,
$E = \{E_1, E_2, E_3, E_4 \}$.
\begin{figure}[!h]
\centering

\scalebox{1} {
\begin{pspicture}(0,-2.4589062)(7.1028123,2.4189062)
\psellipse[linewidth=0.04,dimen=outer](5.2609377,0.14890625)(0.64,2.25)
\psellipse[linewidth=0.04,dimen=outer](1.6909375,0.14890625)(0.61,2.27)
\psellipse[linewidth=0.04,dimen=outer](3.4909375,1.54310626)(2.53,0.49)
\psellipse[linewidth=0.04,dimen=outer](3.5309374,-1.5410937)(2.43,0.42)
\psdots[dotsize=0.12](1.5609375,1.54310626)
\psdots[dotsize=0.12](5.3609376,1.54310626)
\psdots[dotsize=0.12](5.3009377,-1.5610938)
\psdots[dotsize=0.12](1.5409375,-1.6010938)
\usefont{T1}{ptm}{m}{n}
\rput(3.7723436,-2.2310936){$E_3$}
\usefont{T1}{ptm}{m}{n}
\rput(6.592344,0.00890625){$E_2$}
\usefont{T1}{ptm}{m}{n}
\rput(0.43234375,0.10890625){$E_4$}
\usefont{T1}{ptm}{m}{n}
\rput(0.3423437,1.8089062){$x_1$ $(0.5, -0.2)$}
\usefont{T1}{ptm}{m}{n}
\rput(0.12234375,-1.6310937){$x_4$ $(0.4, -0.5)$}
\usefont{T1}{ptm}{m}{n}
\rput(3.6523438,2.2289062){$E_1$}
\usefont{T1}{ptm}{m}{n}
\rput(6.502344,2.1889062){$x_2$ $(0.6, -0.1)$}
\usefont{T1}{ptm}{m}{n}
\rput(6.502344,-1.8710938){$x_3$ $(0.5, -0.4)$}
\end{pspicture}
}
\caption{Bipolar fuzzy hypergraph}
\end{figure}
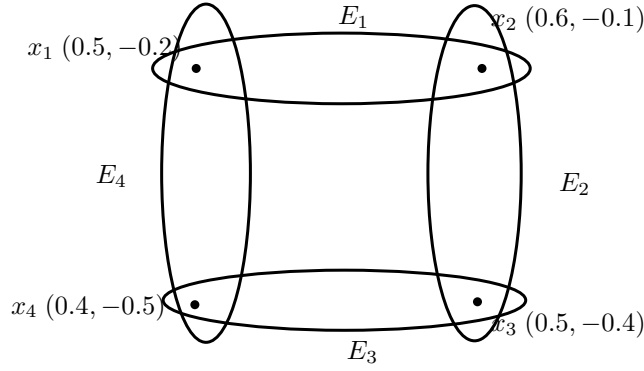
\end{example}
It is easy to see that $E_1$ is strong than $E_3$, and $E_2$ is strong than $E_4$. We call the edges with high strength the strong edges because
the cohesion in them is strong.

\begin{definition}
The {\it $(\alpha, \beta)-$cut of a bipolar fuzzy hypergraph} $H$, denoted by $H_{(\alpha, \beta)}$, is defined as an ordered pair
\[
H_{(\alpha, \beta)}=(V_{(\alpha, \beta)}, E_{(\alpha, \beta)}),
\]
where
\begin{itemize}
  \item [\rm(i)]  $V_{(\alpha, \beta)}=\{x_1, x_2, \ldots, x_n  \}=V,$
  \item [\rm(ii)]$E_{j(\alpha, \beta)}=\{x_i~|~\mu^P_j(x_i)\geq \alpha ~ {\rm  and} ~\mu^N_j(x_i)\leq \beta, ~j=1, 2,3,\ldots, m  \}$,
  \item [\rm(iii)] $E_{m+1(\alpha, \beta)}=\{x_i~|~\mu^P_j(x_i) < \alpha ~ {\rm  and} ~\mu^N_j(x_i) > \beta, ~ \forall ~j  \}.$
\end{itemize}
The edge $E_{m+1(\alpha, \beta)}$ is added to the group of elements which are not contained
in any edge $E_{j(\alpha, \beta)}$ of $H_{(\alpha, \beta)}.$  The edges in the                                $(\alpha, \beta)-$cut  hypergraph are now crisp sets.
\end{definition}
\begin{example}
Consider the bipolar fuzzy hypergraph $H=(V, E)$, where $V=\{x, y, z\}$ and $E=\{E_1, E_2\},$ which is represented by the following incidence matrix:
\begin{table}[!h]
\caption{Incidence matrix of $H$ }
\centering
\medskip
\begin{tabular}{l|cc}

$H$& $E_1$&$E_2$ \\
 \hline
$x$ & $(0.4, -0.2)$ & (0, 0) \\
$y$& $(0.5, -0.3) $& $(0.6, -0.2)$\\
$z$& (0, 0)& $(0.2, -0.05)$\\
\end{tabular}
\end{table}
\newline
From this matrix  we understand that, for example, $E_1=(\mu^P_1, \mu^N_1): V \to [0, 1] \times [-1, 0]$ satisfies:
\[
\mu^P_1(x)=0.4, \ \mu^N_1(x)=-0.2; \ ~ \mu^P_1(y)=0.5, \ \mu^N_1(y)=-0.3; \ ~ \mu^P_1(z)=0, \ \mu^N_1(z)=0.
\]
$(0.3, -0.1)$-cut of  bipolar fuzzy hypergraph $H$ is
\[E_{1(0.3, -0.1)}=\{x, y\},~ E_{2(0.3, -0.1)}=\{y\},~ E_{3(0.3, -0.1)}=\{z\}. \]
The incidence matrix of $H_{(0.3, -0.1)}$ is given below.
\begin{table}[!h]
\caption{Incidence matrix of $H_{(0.3, -0.1)}$ }
\centering
\medskip
\begin{tabular}{l|ccc}

$H_{(0.3, -0.1)}$& $E_{1(0.3, -0.1)}$&$E_{2(0.3, -0.1)}$ & $E_{3(0.3, -0.1)}$\\
 \hline
$x$ &1 &  0 & 0 \\
$y$&1&1 & 0\\
$z$&  0&0 & 1\\
\end{tabular}
\end{table}
\end{example}
%%%%%%%%%%%%%%%%%%%%%%%%%%%%%%%%%%%%%%%%%%%%%%%%%%%%%%%%%%%%%%%%%%%%%%%%%%%%%%%%%%%%%%%%%%%%%%%%%%%%%%%%%%%%%%%%%%%%%%%%%%%%%%%%%%
\begin{definition}
A bipolar fuzzy hypergraph $H=(V, E)$ is
called a $A=(\mu^P_A, \mu^N_A)$-tempered bipolar fuzzy hypergraph of $H=(V, E)$ if there is
a crisp hypergraph $H^*=(V, E^*)$ and a bipolar fuzzy set
$A=(\mu^P_A, \mu^N_A): V \to [0, 1] \times [-1, 0]$ such that $E=\{B_F=(\mu^P_{B_F}, \mu^N_{B_F})~|~ F \in E^*\}$, where

\[
\begin{array}{llll}
\mu^P_{B_F}(x)=\left\{
\begin{array}{lll}
\min (\mu^P_A(y)~|~ y \in F) & {\rm if }~  x \in F, \\[2pt]
0& {\rm  otherwise },
 \end{array} \right.
 \end{array}~~~~~
\begin{array}{llll}
\mu^N_{B_F}(x)=\left\{
\begin{array}{lll}
\max (\mu^N_A(y)~|~ y \in F) & {\rm if }~  x \in F, \\[2pt]
-1& {\rm  otherwise }.
 \end{array} \right.
 \end{array}\]

Let  $A \otimes H$ denote the $A$-tempered bipolar fuzzy hypergraph
of $H$ determined by the crisp hypergraph
$H = (V, E^*)$ and the bipolar fuzzy set $ A:V  \to [0, 1] \times [-1, 0]$.
\end{definition}
\begin{example}
Consider the bipolar fuzzy hypergraph $H=(V, E)$, where $V=\{a, b, c, d\}$ and $E=\{E_1, E_2, E_3, E_4\}$ which is represented by the following incidence matrix:
\begin{table}[!h]
\caption{Incidence matrix of $H$}
\centering
\medskip
\begin{tabular}{l|cccc}
$H$& $E_1$&$E_2$ & $E_3$ & $E_4$  \\
 \hline
$a$ & $(0.2, -0.7)$ & (0, 0) & (0, 0) & $(0.2, -0.7)$ \\
$b$& $(0.2, -0.7)$ & $(0.3, -0.4)$ & $(0.0, -0.9)$ & (0, 0)\\
$c$& (0, 0)& (0, 0) & $(0, -0.9)$ & $(0.2, -0.7)$ \\
$d$& (0, 0)& $(0.3, -0.4)$ & (0, 0) &(0, 0) \\
\end{tabular}
\end{table}
\newline
Then
\[ E_{(0, -0.1)}=\{ \{ b, c\}   \},~  E_{(0.2, -0.7)}=\{ \{a, b\}, \{a, c\}, \{ b, c\}\}.\]
\[
E_{(0.3, -0.1)}=\{\{a, b\}, \{a, c\}, \{ b, c\}, \{ b,  d\}  \}.
\]
Define $A=(\mu^P_A, \mu^N_A): V \to [0, 1] \times [-1, 0]$ by
\[  \mu^P_A(a)=0.2, ~\mu^P_A(b)=\mu^P_A(c)=0.0,~ \mu^P_A(d)=0.3,  \mu^N_A(a)=-0.7, ~\mu^N_A(b)=\mu^N_A(c)=-0.9,~ \mu^N_A(d)=-0.4.\]
Note that
\[ \mu^P_{B_{\{a, b\}}}(a)=\min (\mu^P_A(a), \mu^P_A(b))=0.0, ~ \mu^P_{B_{\{a, b\}}}(b)=\min (\mu^P_A(a), \mu^P_A(b))=0.0,
 \mu^P_{B_{\{a, b\}}}(c)=0.0, ~ \mu^P_{B_{\{a, b\}}}(d)=0.0,\]
\[ \mu^N_{B_{\{a, b\}}}(a)=\max (\mu^N_A(a), \mu^N_A(b))=-0.9, ~ \mu^N_{B_{\{a, b\}}}(b)=\max (\mu^N_A(a), \mu^N_A(b))=-0.9, \mu^N_{B_{\{a, b\}}}(c)=-1, ~ \mu^N_{B_{\{a, b\}}}(d)=-1.\]
Thus \[E_1=(\mu^P_{B_{\{a, b\}}}, \mu^P_{B_{\{a, b\}}}), ~~E_2=(\mu^P_{B_{\{b, d\}}}, \mu^N_{B_{\{b, d\}}}),~
E_3=(\mu^P_{B_{\{b, c\}}}, \mu^N_{B_{\{ b, c\}}}), ~~E_4=(\mu^P_{B_{\{a, c\}}}, \mu^N_{B_{\{a, c\}}}).\]
Hence $H$ is $A$-tempered  hypergraph.
\end{example}

\begin{theorem}
A bipolar fuzzy hypergraph $H$ is a $A=(\mu^P_A, \mu^N_A)$-tempered bipolar fuzzy hypergraph  of some crisp
hypergraph $H^*$ if and only if $H$ is elementary, support simple and simply ordered.
\end{theorem}
\begin{proof}
Suppose that $H=(V, E)$ is a $A$-tempered bipolar fuzzy hypergraph  of some crisp
hypergraph $H^*$. Clearly, $H$ is elementary and support simple. We show that $H$ is simply ordered. Let
 \[C(H)=\{(H^*_1)^{r_1}=(V_1, E_1^*),~ (H^*_2)^{r_2}=(V_2, E_2^*), \cdots,~ (H^*_n)^{r_n}=(V_n, E_n^*)  \}.\]
 Since $H$ is elementary, it follows from Proposition 3.16 that $H$ is ordered. To show that $H$ is simply ordered, suppose that there exists $F \in E^*_{i+1}\setminus E^*_i$. Then there exists $x^* \in F$ such that $\mu^P_A(x^*)=r_{i+1}$, $\mu^N_A(x^*)=\acute{r}_{i+1}$. Since $\mu^P_A(x^*)=r_{i+1}< r_i$ and  $\mu^N_A(x^*)=\acute{r}_{i+1}< \acute{r}_i$, it follows that $x^* \notin V_i$ and $F\nsubseteq V_i$, hence $H$ is simply ordered.\\
 Conversely, suppose $H=(V, E)$ is elementary, support simple and simply ordered. Let
\[C(H)=\{(H^*_1)^{r_1}=(V_1, E_1^*),~ (H^*_2)^{r_2}=(V_2, E_2^*), \cdots,~ (H^*_n)^{r_n}=(V_n, E_n^*)  \}\]
where $D(H)=\{ r_1, r_2, \cdots, r_n\}$ with $0<r_n<\cdots <r_1$. Since $(H^*)^{r_n}=H^*_n=(V_n, E^*_n)$ and define
$A=(\mu^P_A, \mu^N_A):V_n \to [0, 1] \times [-1, 0]$  by

\[
\begin{array}{llll}
\mu^P_{A}(x)=\left\{
\begin{array}{lll}
r_1  &{if }~  x \in V_1, \\[2pt]
r_i & {if }~x \in V_i\setminus V_{i-1}, i=1, 2, \cdots, n
 \end{array} \right.
 \end{array}~~~~~~~
\begin{array}{llll}
\mu^N_{A}(x)=\left\{
\begin{array}{lll}
s_1  &{if }~  x \in V_1, \\[2pt]
s_i & {if }~x \in V_i\setminus V_{i-1}, i=1, 2, \cdots, n
 \end{array} \right.
 \end{array}\] We show that
$E=\{B_F=(\mu^P_{B_F}, \mu^N_{B_F})~|~ F \in E^*\}$, where

\[
\begin{array}{llll}
\mu^P_{B_F}(x)=\left\{
\begin{array}{lll}
\min (\mu^P_A(y)~|~ y \in F) & {\rm if }~  x \in F, \\[2pt]
0& {\rm  otherwise },
 \end{array} \right.
 \end{array}~~~~~~~~
\begin{array}{llll}
\mu^N_{B_F}(x)=\left\{
\begin{array}{lll}
\max (\mu^N_A(y)~|~ y \in F) & {\rm if }~  x \in F, \\[2pt]
-1& {\rm  otherwise }.
 \end{array} \right.
 \end{array}\]
 Let $F \in E^*_n$. Since $H$ is  elementary and support simple, there is a unique bipolar fuzzy edge $C_F=(\mu^P_{C_F}, \mu^N_{C_F})$ in $E$ having support $E^*$. Indeed, distinct edges in $E$ must have distinct supports that lie in  $E^*_n$. Thus, to show that $E=\{ B_F=(\mu^P_{B_F}, \mu^N_{B_F})~|~F \in E^*_n\}$,  it suffices to show that for each $F \in E^*_n$,
 $\mu^P_{C_F}=\mu^P_{B_F}$ and $\mu^N_{C_F} =\mu^N_{B_F}$. As all edges are elementary and different edges have different supports, it follows from the definition of fundamental sequence that $h(C_F)$ is equal to some number $r_i$ of $D(H)$. Consequently, $E^* \subseteq V_i $. Moreover, if $i>1$, then $F \in E^*\setminus E^*_{i-1}$. Since $F\subseteq  V_i$, it follows from
 the definition of $A=(\mu^P_A, \mu^N_A)$ that for each $x \in F$, $\mu^P_A(x)\geq r_i$ and $\mu^N_A(x)\leq s_i$. We claim that
 $\mu^P_A(x)= r_i$ and $\mu^N_A(x)= s_i$, for some $x\in F$. If not, then by definition of $A=(\mu^P_A, \mu^N_A)$, $\mu^P_A(x)\geq r_i$ and $\mu^N_A(x)\leq s_i$ for all $x \in F$ which implies that $F \subseteq  V_{i-1}$ and so $F \in E^*\setminus E^*_{i-1}$ and since $H$ is simply ordered $F\varsubsetneq V_{i-1}$, a contradiction. Thus it follows from the definition of $B_F$ that $B_F=C_F$. This completes the proof.
\end{proof}
As a consequence of the above theorem we obtain.
\begin{proposition}
Suppose that $H$ is a simply ordered bipolar fuzzy hypergraph and $F(H)=\{r_1, r_2, \cdots, r_n \}$. If $H^{r_n}$ is a simple hypergraph, then there is a partial bipolar fuzzy hypergraph $\acute{H}$ of $H$ such that the following assertions hold:
\begin{itemize}
  \item[\rm (1)] $\acute{H}$ is a $A=(\mu^P_A, \mu^N_A)$-tempered bipolar fuzzy hypergraph  of $H_n$.
  \item [\rm (2)]$E \sqsubseteq \acute{E}$.
  \item [\rm (3)] $F(\acute{H})=F(H)$  and $C(\acute{H})=C(H).$
\end{itemize}
\end{proposition}
\section{Application examples of  bipolar  fuzzy hypergraphs}
 
\begin{definition} Let $X$ be a reference set. Then, a family of nontrivial bipolar fuzzy sets $\{A_1, A_2, A_3, \ldots, A_m \}$  where $A_i=(\mu^P_i, \mu^N_i)$ is a {\it bipolar  fuzzy partition} if
\begin{itemize}
  \item [\rm (1)]  $\bigcup_i {\rm supp}(A_i)=X$, ~~ $i=1, 2, \ldots, m$,
    \item [\rm (2)] $\sum^m_{i=1} \mu^P_i(x)=1$  for all $x \in X,$
    \item [\rm (3)] $\sum^m_{i=1} \mu^N_i(x)=-1$  for all $x \in X.$
\end{itemize}
\end{definition}
Note that this definition generalizes fuzzy partitions because
the definition is equivalent to a fuzzy partition when
for all $x$,  $\nu_i(x)$ =0.  We call a family $\{A_1, A_2, A_3, \ldots, A_m \}$ a {\it bipolar fuzzy covering} of $X$ if it satisfies above conditions $(1)-(3)$.

A bipolar fuzzy partition can be represented by a bipolar fuzzy matrix $[a_{ij}]$ where
$a_{ij}$, is the positive membership degree and negative membership degree of element $x_i$ in class $j.$ We see that the
matrix is the same as the incidence matrix in bipolar fuzzy hypergraph. Then
we can represent a bipolar  fuzzy partition by a bipolar  fuzzy hypergraph  $H = (V, E)$
such that\\
$(1)$ \  $V$: a set of elements $x_i$, $i$ = 1, $\ldots$, $n$\\
$(2)$ \ $E=\{E_1, E_2, \ldots, E_m\}$= a set of nontrivial bipolar fuzzy classes, \\
$(3)$ \ $V = \bigcup_{j}  {\rm supp}(E_j),~~ j=1, 2, \ldots, m,$\\
$(4)$ \   $\sum^m_{i=1} \mu^P_i(x)=1$   for all $x \in X$,\\
$(5)$ $\sum^m_{i=1} \mu^N_i(x)=-1$   for all $x \in X$.

Note that conditions (4) - (5) are added to the
bipolar  fuzzy hypergraph for bipolar  fuzzy partition. If these conditions are added, the
bipolar  fuzzy hypergraph can represent a bipolar  fuzzy covering. Naturally, we can
apply the $(\alpha, \beta)$-cut to the bipolar  fuzzy partition.\\

\begin{example}[ Radio coverage network]
In telecommunications, the coverage of a radio station is the geographic area where the station can communicate.\\
 Let $V$ be a finite set of radio
receivers (vertices); perhaps a set of representative locations at the
centroid of a geographic region. For each of $m$ radio transmitters we define the bipolar fuzzy set ``listening area of station $j$" where $A_j(x)=(\mu^P_{Aj}(x), \mu^N_{Aj}(x))$ represents
the ``quality of reception of station $j$ at location $x$." Positive membership  value near to $1$,  could signify
``very clear reception on a  very sensitive
radio"  while  negative membership  value
near to $-1$,  could signify  ``very poor reception on  a very poor radio". Since graphy affects signal strength, each ``listening area" is an bipolar fuzzy set. Also, for a fixed radio the reception will vary between
different stations.  Thus this model  uses the full definition of an bipolar
fuzzy hypergraph. The model could be used to determine station programming
or marketing strategies  or  to establish an emergency broadcast network. Further
variables could relate signal strength to changes in time of day, weather
and other conditions.\end{example}
\begin{example}[Clustering problem]
We consider here the clustering problem, which is a typical example of a bipolar  fuzzy partition on the digital image processing.\\
There are five objects and they are classified into two classes: tank
and house. To cluster the elements $x_1$, $x_2$, $x_3$, $x_4$, $x_5$ into $A_t$ (tank) and
$B_h$ (house), a bipolar fuzzy partition matrix is given as the form of incidence
matrix of bipolar fuzzy hypergraph.

 \begin{table}[!h]
\caption{Bipolar fuzzy partition matrix }
\centering
\medskip
\begin{tabular}{l|cc}
$H$& $A_t$ & $B_h$ \\
 \hline
$x_1$ & $(0.96, -0.04)$ & $(0.04, -0.96)$  \\
$x_2$&($0.95, -0.5)$& $(0.5, -0.95)$ \\
$x_3$&  $(0.61, -0.39)$& $(0.39, -0.61)$ \\
$x_4$&  $(0.05, -0.95)$& $(0.95, -0.05)$ \\
$x_5$&  $(0.03, -0.97)$& $(0.97, -0.03)$ \\
\end{tabular}
\end{table}
 We can apply the $(\alpha, \beta)$-cut to the hypergraph and obtain a hypergraph $H_{(\alpha, \beta)}$ which is not bipolar fuzzy hypergraph.  We denote the edge (class) in
$(\alpha, \beta)$-cut hypergraph $H_{(\alpha, \beta)}$ as $E_{j(\alpha, \beta)}$. This hypergraph $H$, represents generally the covering because the conditions: $(4)$~ $\sum^m_{i=1} \mu^P_i(x)=1$ for all $x \in X$, $(5)$$\sum^m_{i=1} \mu^N_i(x)=-1$ for all $x \in X$, is not
always guaranteed. The hypergraph $H_{(0.61, -0.03 )}$ is shown in Table $8$.
\begin{table}[!h]
\caption{ Hypergraph $H_{(0.61, -0.03)}$ }
\centering
\medskip
\begin{tabular}{c|cc}
$H_{(0.61, -0.03)}$& $A_{t(0.61, -0.03)}$&$B_{h(0.61, -0.03)}$\\
 \hline
$x_1$ &1 &  0 \\
$x_2$&1&0 \\
$x_3$&  1&0 \\
$x_4$&  0&1 \\
$x_5$&  0&1 \\
\end{tabular}
\end{table}

We  obtain dual bipolar fuzzy hypergraph $H^D_{(0.61, -0.03)}$ of $H_{(0.61, -0.03)}$ which is given in Table $9$.
\begin{table}[!h]
\caption{Dual bipolar fuzzy hypergraph}
\centering
\medskip

\begin{tabular}{c|ccccc}
$H^D(0.61, -0.03)$& $X_1$&$X_2$ &$X_3$ & $X_4$ & $X_5$\\
 \hline
$A_t$ &1 & 1& 1&0 & 0 \\
$B_h$& 0&0& 0& 1 & 1\\
\end{tabular}
\end{table}

We consider the strength of edge (class) $E_{j(\alpha,\beta)}$, or in the $(\alpha, \beta)$-cut
hypergraph $H_{(\alpha, \beta)}$. It is necessary to apply Definition \ref{D2}
to obtain the strength of edge $E_{j(\alpha, \beta)}$ in $H_{(\alpha,\beta )}$.\\
The possible interpretations of $\eta(E_{j(\alpha,\beta)})$ are:
\begin{itemize}
\item the edge (class) in the hypergraph (partition) $H_{(\alpha, \beta)}$, groups elements
having at least $\eta$ positive and negative memberships,
\item the strength (cohesion) of edge (class) $E_{j(\alpha, \beta)}$  in $H_{(\alpha, \beta)}$ is $\eta$.
  \end{itemize}

 Thus we can use the strength as a measure of the cohesion or
strength of a class in a partition. For example, the strengths of classes $A_t(0.61, -0.03)$ and $B_h(0.61, -0.03)$ at $s$=0.61, $t$=$-0.03$ are $\eta(A_t(0.61, -0.03))$=$(0.96, -0.04)$, $\eta(B_h(0.61, -0.03))$=$(0.97, -0.03).$ Thus we say that the class $\eta(B_h(0.61, -0.03))$ is stronger than $\eta(A_t(0.61, -0.03))$ because $\eta(B_h(0.61, -0.03)) >\eta(A_t(0.61, -0.03))$. From the above discussion on the hypergraph
$H_{(0.61, -0.03)}$ and $H^D_{(0.61, -0.03)}$, we can state that:

\begin{itemize}
  \item The bipolar fuzzy hypergraph can represent the fuzzy partition visually.
The $(\alpha, \beta)$-cut hypergraph also represents the $(\alpha, \beta)$-cut partition.
    \item The dual hypergraph $H^D_{(0.61, -0.03)}$ can represent elements $X_i$, which can
be grouped into a class $E_{j(\alpha, \beta)}$. For example, the edges $X_1$, $X_2$, $X_3$ of the dual hypergraph in Table 9 represent that the elements $x_1$, $x_2$, $x_3$  that can be grouped into $A_t$ at level (0.61, -0.03).

\item At $(\alpha, \beta)$=(0.61, -0.03) level, the strength of class $B_h(0.61, -0.03)$ is the highest
(0.95, -0.05), so it is the strongest class. It means that this class can
be grouped independently from the other parts. Thus we can
eliminate the class $B_h$ from the others and continue clustering.
Therefore, the discrimination of strong classes from the others
can allow us to decompose a clustering problem into smaller
ones. This strategy allows us to work with the reduced data in
a clustering problem.
\end{itemize}
\end{example}

\end{document}